\documentclass[reqno]{amsart}
\usepackage{amsmath,amssymb,amsthm}
\usepackage[cmtip,color,all]{xy}
\xyoption{tips}
\usepackage{parskip}
\SelectTips{cm}{10}
\usepackage[dvipsnames,table,xcdraw]{xcolor}
\usepackage{enumitem}
\usepackage{mathtools}
\usepackage{multirow}
\usepackage{microtype}
\usepackage{booktabs}
\usepackage{caption}
\usepackage{diagbox}
\usepackage{mathrsfs}
\usepackage{hyperref}
\usepackage{tikz-cd}
\usetikzlibrary{arrows.meta}
\usetikzlibrary{shapes.geometric}
\usepackage{bbold}
\usepackage{floatrow}
\usepackage{pifont}
\usepackage{comment}
\usepackage{geometry}
\usepackage{setspace}
\usepackage{stmaryrd,xcolor,soul,tikz}
\usepackage{pdfpages}
\usepackage{multirow}

\definecolor{gold}{rgb}{1.0, 0.87, 0.4}

\newcommand{\C}{\mathcal{C}} 
\newcommand{\T}{\mathcal{T}} 
\newcommand{\K}{\mathcal{K}}

\newcommand{\AF}{\mathsf{AF}}
\newcommand{\AC}{\mathsf{AC}}
\newcommand{\W}{\mathsf{W}}

\newcommand{\F}{\mathsf{F}}
\newcommand{\Cof}{\mathsf{C}}

\usepackage{tikz}
\usetikzlibrary{positioning,arrows}
\usetikzlibrary{shapes.geometric, calc}
\usetikzlibrary{decorations.pathreplacing}

\newtheorem{theorem}{Theorem}[section]
\newtheorem{corollary}[theorem]{Corollary}
\newtheorem{lemma}[theorem]{Lemma}
\newtheorem{proposition}[theorem]{Proposition}

\newtheorem*{theorem*}{Theorem}

\theoremstyle{definition}
\newtheorem{definition}[theorem]{Definition}
\newtheorem{example}[theorem]{Example}
\newtheorem{remark}[theorem]{Remark}

\newtheorem{construction}[theorem]{Construction}
\usepackage{hyperref}

\definecolor{dark-red}{rgb}{0.5,0.15,0.15}
\definecolor{dark-blue}{rgb}{0.15,0.15,0.6}
\definecolor{dark-green}{rgb}{0.15,0.6,0.15}
\hypersetup{
    colorlinks, linkcolor=dark-red,
    citecolor=dark-blue, urlcolor=dark-green
}

\definecolor{gRed}{HTML}{ff5100}
\definecolor{gGreen}{HTML}{2b83ba}

\usepackage{fancyhdr}           
\fancypagestyle{plain}{
    \lhead{}
    \fancyhead[R]{\thepage}
    \fancyhead[L]{}
    
    \fancyfoot{}
}

\usepackage[textwidth=3cm, textsize=small, colorinlistoftodos]{todonotes}

\allowdisplaybreaks
\graphicspath{{./images/}}

\title{Bousfield Localizations on the Nonmodular Lattice $N_5$}
\author{Sofía Martínez Alberga}
\address{Bryn Mawr College, Bryn Mawr, PA, USA}
\email{smartineza@brynmawr.edu}
\author{Constanze Roitzheim}
\address{University of Kent, SMSAS, Canterbury CT2 7FS, UK}
\email{c.roitzheim@kent.ac.uk}
\date{\today}

\begin{document}
\begin{abstract}
We provide a complete description of the model category structures on the nonmodular lattice $N_5$. 
Furthermore we explain how these model category structures are related to each other via Bousfield localization. 
This work heavily relies on the use of combinatorical objects from equivariant homotopy theory known as \emph{transfer systems}, and it results in a wealth of interesting interactions between combinatorial and homotopical methods.
\end{abstract}
\maketitle
\section{Introduction}
Model category structures supply us with a fundamental language for homotopy theory. Since their original inception in the 1960's, they have become invaluable for formalizing phenomena in rational, stable, and equivariant homotopy, which classically requires extensive background knowledge about the specific underlying categories of e.g. spaces or spectra. If instead one considers model category structures on a finite lattice, one strips model categories from those elaborate technical prerequisites and brings the model category techniques themselves into the foreground. 

To discuss model structures on a finite lattice $\mathcal{P}$, \emph{transfer systems} play an important role. A transfer system is a subcategory of a lattice that contains the all objects of $\mathcal{P}$ and is closed under pullbacks.
These special subcategories prove to be a central tool for the study of equivariant homotopy commutativity \cite{Blumberg-Hill}, and also link to other rich combinatorial structures \cite{BBR} \cite{BOOR}.
A model category structure on $\mathcal{P}$ determines and is determined by a pair $\W \supseteq \mathcal{T}$, where $\W$ is a wide decomposable subcategory and $\mathcal{T}$ is a {transfer system}, subject to further restrictions specified in \cite{MORSVZ2}. 
In this case, $\W$ plays the role of the weak equivalences and $\mathcal{T}=\AF$ the role of the acyclic fibrations of the model structure.  
Therefore, not only do transfer systems on finite lattices provide bridges between seemingly unrelated concepts, they also shift the focus away from abstract set theory to more manageable, ``hands-on'' constructions.

One such construction on model categories is \emph{Bousfield localization}, which provides a formal framework for modifying an existing model structure into a new model structure by incrementing the class of weak equivalences. 
A common occurrence of Bousfield localization in topology is localizing a category of topological spaces or spectra by changing the weak equivalences from weak homotopy equivalences to $E_*$-isomorphisms for a suitable homology theory $E_*$.
In the context of finite lattices, the effect of left and right localization on a pair $(\W,\AF)$ has been determined in \cite{CGMNRT}. A part of this work shows that while left Bousfield localization of course leaves $\AF$ unchanged, right Bousfield localizations on $\mathcal{P}$ are classified by a minimal generating arrows called \emph{Golden Arrows} which depend only on $\W$. The present article exemplifies another instance of Golden Arrows on a lattice that is not rectangular or even modular.

The specific lattice $N_5$ considered in this paper is relatively small, but is \emph{nonmodular}, meaning that it fails to satisfy crucial distributive properties.
While nonmodular lattices are harder to work with in practice than their nicely-behaved modular counterparts, they do make a frequent appearance in homotopy theory.
To start, in homotopy-coherent structures, the lattice $N_5$ appears as the $3$-\emph{Tamari lattice}.
Recall that the $n$-Tamari lattice this is the poset consisting of sequences of $(n+1)$-letters, where the order is induced by applications of the associative law $((xy)z) \rightarrow (x(yz))$. 
Furthermore the $n$-Tamari lattice is isomorphic to the poset of transfer systems on a total order $[n]$, see \cite{BBR}. As a direct consequence, if a finite lattice $\mathcal{P}$ contains $[2]$, its poset of transfer systems $\mathrm{Tr}(\mathcal{P})$ will always be nonmodular. 
Additionally it was exhibited by \cite{KLMSV} that the subgroup lattice Sub($G$) is nonmodular for numerous interesting classes of nonabelian groups $G$ such as dihedral groups and alternating groups. 
Therefore, in order to understand $G$-equivariant stable homotopy theory for these groups, it is vital to understand transfer systems on $N_5$ and their homotopy-theoretic interactions. 

Our main result is the following, which appears as Theorem \ref{thm:allthelocals}.

\begin{theorem*}
Every model category structure on $N_5$ can be obtained as a sequence of left and right Bousfield localizations of the trivial model structure.
\end{theorem*}

This new understanding of the behaviour of Bousfield localizations on $N_5$ will add a useful puzzle piece to the behaviour of Bousfield localizations on finite lattices in general.

\subsection*{Organization}
This paper is organized as follows. In Section \ref{sec:lattices} we introduce $N_5$ and recall some basic lattice terminology. Section \ref{sec:transfers} develops the lattice of transfer systems on $N_5$, as well as the dual notion known as cotransfer systems. In this section we will see that the fact that pushouts of indecomposable arrows are not necessarily indecomposable again is the first major practical obstacle of working with a nonmodular lattice. Section \ref{sec:models} recalls the relationship between transfer systems and model structures, which allows us to systematically determine all model structures on $N_5$ as well as draw some first conclusions from this. Finally, in Section \ref{sec:bousfield} we explore the effect of Bousfield localizations on our previous calculations, which also demonstrates the application of the Golden Arrows in a non-modular context.

\subsection*{Acknowledgments}
The first author would like to thank Bryn Mawr College of the institutional support. 
The second author would like to thank the organizers of the Abel Symposium 2025 for the invitation and travel support.
\section{Nonmodular lattices}\label{sec:lattices}

In this section we introduce the lattice $N_5$ and recall some of its properties that we will require in later sections.

\begin{definition}
A \emph{lattice} is a poset such that each two objects have a unique \emph{meet} and \emph{join} in the following sense.
\begin{itemize} 
\item Meet: $X \vee Y$ is the smallest $Z$ with $X \leq Z$ and $Y \leq Z$,
\item Join: $X \wedge Y$ is the largest $Q$ with $Q \leq X$ and $Q \leq Y$.
\end{itemize} 
\end{definition}

We will use the notation $X \leq Y$ and $X \rightarrow Y$ interchangeably. Furthermore,
in a lattice, any square of the form 
\[
\xymatrix{ X \ar[r]  & X \vee Y \\
A \ar[r] \ar[u] & Y \ar[u]
}
\]
is always a pushout square, and
\[
\xymatrix{ X \ar[r]  & B \\
X \wedge Y \ar[r] \ar[u] & Y \ar[u]
}
\]
is always a pullback square.

The lattice $N_5$ is the pentagon depicted in the Figure \ref{fig:introducingn5}.
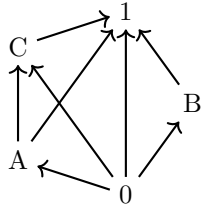
\begin{figure}[H]
    \centering
    \begin{tikzpicture}    
\node (X) at (1.4265848,-0.4635254) {0};
\node (B) at (2.3082627, 0.75) {B};
\node (A) at (0,0) {A};
\node (C) at (0,1.5) {C};
\node (Y) at (1.4265848,1.9635254) {1};
\draw[thick, ->](X) edge (A);
\draw[thick, ->](X) edge (B);
\draw[thick, ->](X) edge (C);
\draw[thick, ->](X) edge (Y);
\draw[thick, ->](A) edge (C);
\draw[thick, ->](A) edge (Y);
\draw[thick, ->](B) edge (Y);
\draw[thick, ->](C) edge (Y);
\end{tikzpicture}
    \caption{The lattice $N_5$.}
    \label{fig:introducingn5}
\end{figure}

Its objects are $0, A, B, C$ and $1$ with $0 < A < C <1$, $0<B<1$ and $B$ is not comparable to $A$ or $C$. 
For simplicity, we will omit the names of the objects in the majority of the subsequent images.
Furthermore we call a morphism $f$ in poset \emph{short} if it is indecomposable, i.e. $f= g \circ h$ implies that either $g$ or $h$ is the identity.
In Table \ref{tab:pushoutspullbacks},
the full list of nontrivial pushouts and pullbacks is given and note that the six short morphisms in $N_5$ are listed first.
\begin{center}
\begin{table}[h]
\begin{tabular}{ |p{4cm}||p{4cm}|p{4cm}|  }
 \hline
 Arrow & Pushout & Pullback \\
 \hline
 \hline
 $0 \rightarrow A$ & $B \rightarrow 1$ & $\emptyset$ \\
  $A \rightarrow C$ & $\emptyset$ & $\emptyset$\\
 $C \rightarrow 1$ & $\emptyset$ &  $0 \rightarrow B$ \\
 $0 \rightarrow B$ & $A \rightarrow 1$, $C \rightarrow 1$ & $\emptyset$ \\
 $B \rightarrow 1$ & $\emptyset$ & $0 \rightarrow A$, $0 \rightarrow C$  \\
 \hline
 $A \rightarrow 1$ & $C \rightarrow 1$ & $0 \rightarrow B$, $A \rightarrow C$ \\
 $0 \rightarrow C$ & $B \rightarrow 1$, $A \rightarrow C$ &$0 \rightarrow A$ \\
 \hline
 $0 \rightarrow 1$ &  $A \rightarrow 1$,  $B \rightarrow 1$, $C \rightarrow 1$ & $0 \rightarrow A$, $0 \rightarrow B$, $0 \rightarrow C$ \\
 \hline
\end{tabular}
\caption{Pushouts and pullbacks in $N_5$}
\label{tab:pushoutspullbacks}
\end{table}
\end{center}
This pentagon is the smallest example of a \emph{nonmodular lattice}. Recall a lattice is \emph{modular} or \emph{distributive} if it does not contain $N_5$ as a sublattice.
The following example illustrates this nuance in detail.

\begin{example}
The lattice $[2] \times [1]$, pictured below,
\begin{figure}[h]
    \centering
    \begin{tikzpicture}[scale=1.1]
 \foreach \x in {1,2, 3}
 \foreach \y in {1,2}
    \fill (\x,\y) circle (1mm);
\node (A) at (1,2) {};
\node (B) at (2,2) {};
\node (C) at (3,2) {};
\node (D) at (1,1) {};
\node (E) at (2,1) {};
\node (F) at (3,1) {};

\draw[thick, ->](A) edge (B);
\draw[thick, dashed, ->](D) edge (E);
\draw[thick, ->](D) edge (C);
\draw[thick, ->](D) edge (A);
\draw[thick, ->](E) edge (F);
\draw[thick, dashed, ->](D) edge[bend right] (F);
\draw[thick, ->](D) edge (B);
\draw[thick, ->](A) edge[bend left] (C);
\draw[thick, dashed,->](B) edge (C);
\draw[thick, ->](E) edge (C);
\draw[thick, ->, dashed,](F) edge (C);
\draw[thick, dashed, ->](E) edge (B);
\end{tikzpicture}
    \label{fig:21}
\end{figure}
contains $N_5$ as a subposet in the form of the dashed arrows, with 
\[
0 \mapsto (0,0),\,\, A \mapsto (0,1),\,\, C \mapsto (1,1),\,\, B \mapsto (2,0),\,\, 1 \mapsto (2,1).
\]
However, this does not make $N_5$ a sublattice as $B \wedge C = 0$ but $(1,1) \wedge (2,0)=(1,0)$.
\end{example}

An equivalent definition of modularity is that a lattice is modular if and only if it satisfies the \emph{modular law} or \emph{distributive law}, namely that
\[
x \leq y \Rightarrow x \vee (a \wedge y) = (x \vee a) \wedge b.
\]

In practice, many proofs and situations assume lattices to be modular, and therefore nonmodular lattice often require special treatment. Before our next remark on modularity, let us introduce the following convenient terminology. 

\begin{definition}
We say that an arrow $f$ in a lattice is \emph{short} if it is indecomposable, i.e. $f=f_2 \circ f_1$ implies that either $f_1$ or $f_2$ is the identity. 
\end{definition}

\begin{remark}
A lattice is modular if and only if the pushout of a short arrow is again short. We see that this is clearly not the case for $N_5$. Conversely, assume that we have a lattice where the arrow $x \rightarrow y$ has a pushout $p \rightarrow q$ that factors as $p \rightarrow z \rightarrow q$. Firstly, this already implies that $y$ and $z$ are not comparable, because otherwise the pushout would not have target $q$. 
Then, even though $p \leq z$, we have
\[
p \vee (y \wedge z) = p \neq z = (p \vee y) \wedge z,
\]
which violates the distributive law.
\end{remark}

\section{Transfer systems}\label{sec:transfers}
Transfer systems are a combinatorial structure naturally arising in equivariant topology and also, as we will see in the next section, in model categories. First, let us recall the definition and some examples, see also \cite[Definition 2.6]{MORSVZ2}. 
\begin{definition}
A \emph{transfer system} on a category $\mathcal{C}$ is a wide subcategory closed under pullbacks. 
If $\mathcal{C}$ is furthermore a finite lattice, then we denote by $\mathrm{Tr}(\C)$ the poset of all transfer systems on $\C$ ordered by inclusion.
\end{definition}

\begin{example}
To visualise this definition, let us look at the total order $[2]$. Let $\T$ be a transfer system on $[2]$. By definition $\T$ is a subcategory and therefore closed under composition. Consequently, if $0\rightarrow 1$ and $1 \rightarrow 2$ are both in $\T$, then so is $0 \rightarrow 2$.
Also by definition $\T$ is closed under pullbacks, so if $0 \rightarrow 2 \in \T$, then $0 \rightarrow 1 \in \T$. Taking those facts into account, one can work out that there are precisely five transfer systems on $\T$ listed in Figure \ref{fig:uretson2}.
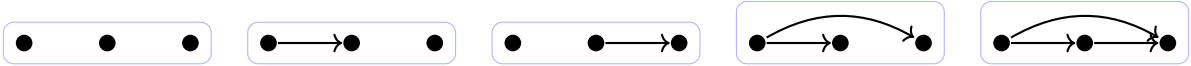
\begin{figure}[h]
\begin{tikzpicture}[scale=1.1]
 \draw[blue!30, rounded corners] (0.75,0.75) rectangle (3.25,1.25);
 \foreach \x in {1,2, 3}
 \foreach \y in {1}
    \fill (\x,\y) circle (1mm);
\node (A) at (1,1) {};
\node (D) at (2,1) {};
\node (H) at (3,1) {};
\end{tikzpicture}
\quad
\begin{tikzpicture}[scale=1.1]
 \draw[blue!30, rounded corners] (0.75,0.75) rectangle (3.25,1.25);
 \foreach \x in {1,2, 3}
 \foreach \y in {1}
    \fill (\x,\y) circle (1mm);
\node (A) at (1,1) {};
\node (D) at (2,1) {};
\node (H) at (3,1) {};
\draw[thick, ->](A) edge (D);
\end{tikzpicture}
\quad 
\begin{tikzpicture}[scale=1.1]

 \draw[blue!30, rounded corners] (0.75,0.75) rectangle (3.25,1.25);
 \foreach \x in {1,2, 3}
 \foreach \y in {1}
    \fill (\x,\y) circle (1mm);
\node (A) at (1,1) {};
\node (D) at (2,1) {};
\node (H) at (3,1) {};
\draw[thick, ->](D) edge (H);
\end{tikzpicture}
\quad
\begin{tikzpicture}[scale=1.1]

 \draw[blue!30, rounded corners] (0.75,0.75) rectangle (3.25,1.5);
 \foreach \x in {1,2, 3}
 \foreach \y in {1}
    \fill (\x,\y) circle (1mm);
\node (A) at (1,1) {};
\node (D) at (2,1) {};
\node (H) at (3,1) {};
\draw[thick, ->](A) edge (D);
\draw[thick, ->](A) edge[bend left]  (H);
\end{tikzpicture}
\quad 
\begin{tikzpicture}[scale=1.1]

 \draw[blue!30, rounded corners] (0.75,0.75) rectangle (3.25,1.5);
 \foreach \x in {1,2, 3}
 \foreach \y in {1}
    \fill (\x,\y) circle (1mm);
\node (A) at (1,1) {};
\node (D) at (2,1) {};
\node (H) at (3,1) {};
\draw[thick, ->](A) edge (D);
\draw[thick, ->](A) edge[bend left]  (H);
\draw[thick, ->](D) edge (H);
\end{tikzpicture}
\caption{The transfer systems on $[2]$}
\label{fig:uretson2}
\end{figure}

\end{example}

\begin{example}
With the information in Table \ref{tab:pushoutspullbacks}, we can work out what the transfer systems on the pentagon $N_5$ are. We display them in Figure \ref{fig:TSN5}. We will learn about the role of the eight transfer systems marked with a red asterisk in the right lower corner later in this section.

\begin{figure}[H]
    \centering
    \scalebox{.75}{\input{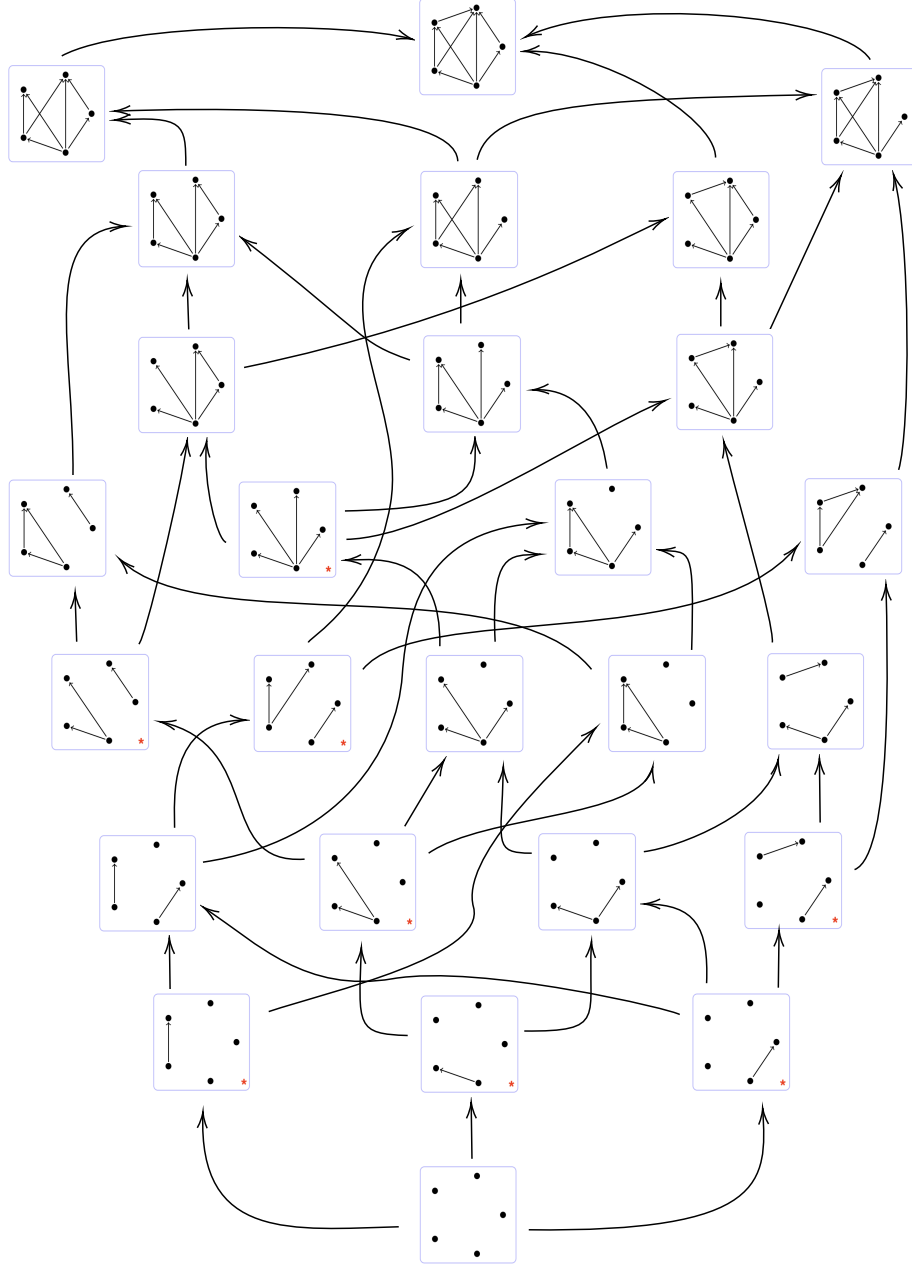}}
    \caption{The 26 transfer systems on $N_5$, ordered by inclusion.}
    \label{fig:TSN5}
\end{figure}

\end{example}

Next, we define the dual concept to a transfer system.

\begin{definition}
A \emph{cotransfer system} on a category $\C$ is a wide subcategory closed under pushouts. 
\end{definition}
 
Before we summarize some duality properties between transfer and cotransfer systems, let us introduce some more notation.

\begin{definition}
Let $S$ be a set of morphisms in $\C$.
We say that a morphisms $f \in \C$ has the \emph{left lifting property (LLP) with respect to $S$} if for every $s \in S$ and for every commutative diagram of the form
\[
\xymatrix{A \ar[r] \ar[d]_f & X \ar[d]^s\\
B \ar[r] & Y 
}
\]
there is a morphism $B \rightarrow X$ (a ``lift'') making the resulting diagram commute. We write 
\[
{}^\boxslash S = \{ f \in \C \,\,|\,\,\mbox{$f$ has the LLP with respect to $S$} \}.
\]
We define the \emph{right lifting property with respect to $S$} and the set $S^\boxslash$ analogously.
\end{definition}

From these definitions, one can conclude the following, see also \cite[Proposition 4.2]{FOOQW}.

\begin{proposition}
The set $\T$ is a transfer system on $\C$ if and only if ${}^\boxslash \T$ is a cotransfer system on $\C$. Dually, $\K$ is a cotransfer system on $\C$ if and only if ${\K}^\boxslash$ is a transfer system on $\C$.
\end{proposition}

In particular, if $\C$ is a finite lattice, there are as many transfer systems as there are cotransfer systems on $\C$.

\begin{example}
Table \ref{fig:transfersandcotransfers} is a list of all transfer systems $\T$ on the pentagon $N_5$, paired up with their respective dual cotransfer system ${}^\boxslash\T$. 
Note that the dual cotransfer system is in general not simply the complement set of arrows. 
We direct the interested reader  to Section 3 of \cite{MORSVZ2} for more details on how to construct a cotransfer system from a transfer system.

\begin{figure}[t]\label{fig:tsandcs}
    \centering
    \scalebox{.5}{\input{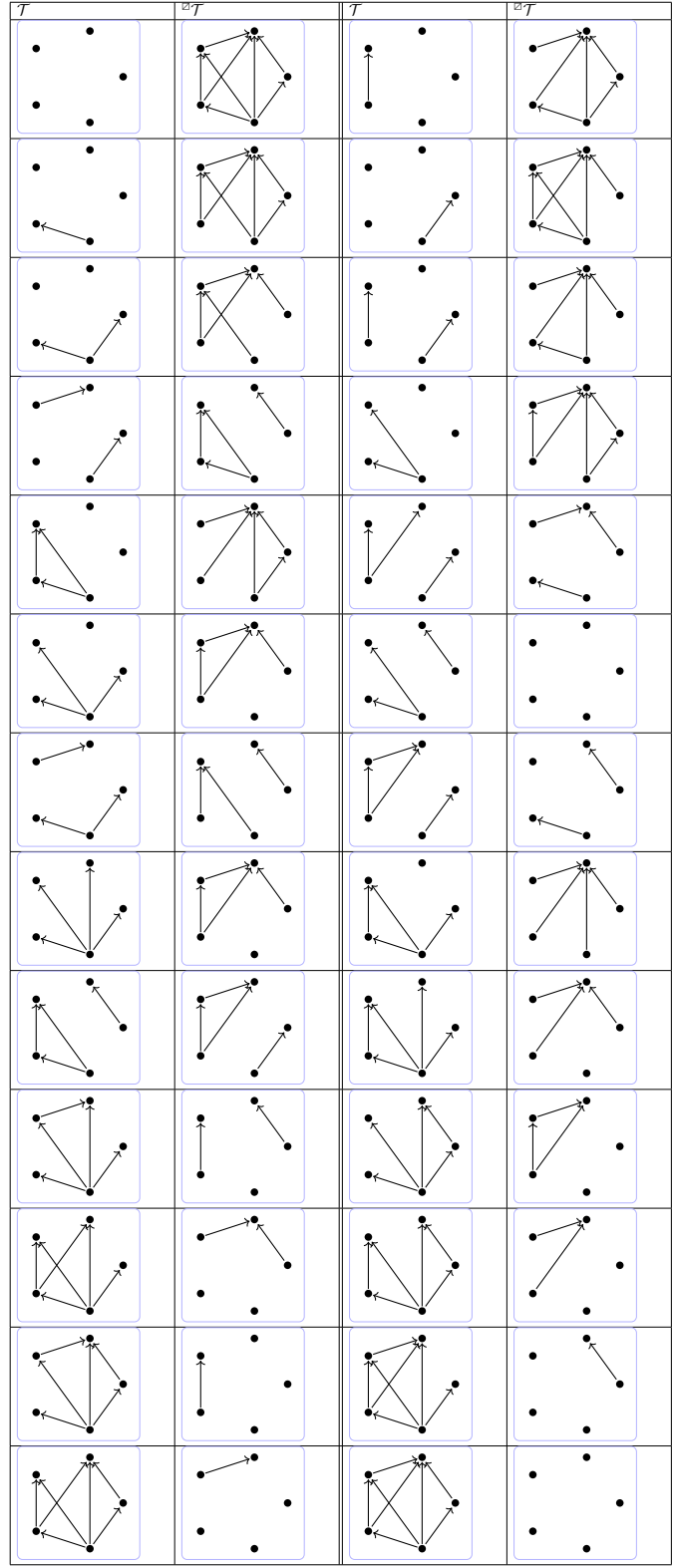}}
    \caption{The 26 transfer systems on $N_5$ with their dual cotransfer systems.}
    \label{fig:transfersandcotransfers}
\end{figure}

\end{example}

We will now introduce the concept of a transfer system generated by a set.
Let $S$ be a set of morphisms in $\C$. Then we define the smallest transfer system containing $S$ to be denoted as $\left<S\right>_{TS}$.
If $\C$ is a finite lattice, $\left<S\right>_{TS}$ always exists as it is the intersection of all transfer systems containing $S$. This intersection both finite and nonempty, as $\T=\C$ is always a transfer system containing $S$. One can make $\left<S\right>_{TS}$ even more explicit, see \cite[Theorem A.2]{Rub21} or \cite[Lemma 3.6]{MORSVZ}, namely as
\[
\left< S \right>_{TS}=\{s_n \circ s_{n-1} \cdots \circ s_1 \,\,|\,\, n \ge 0, \,\,\mbox{$s_i$is a pullback of an element of $S$} \}.
\]
In other words, we first close $S$ under pullbacks and then we close the result of this under composition.

If $\C$ is a lattice, then the poset of transfer systems $\mathrm{Tr}(\C)$ on $\C$ forms a lattice itself \cite[Theorem 2.20]{MORSVZ2}. The meet operation $T_1 \vee T_2$ is given by the intersection $T_1 \cap T_2$, and the join of two transfer systems is given by the smallest transfer system containing both $T_1$ and $T_2$, i.e.
\[
T_1 \wedge T_2 = T_1 \cap T_2 \,\,\,\mbox{and}\,\,\, T_1 \vee T_2 = \left<T_1 \cup T_2\right>_{TS}.
\]

\begin{example}
In the example of a total order $[n]$, the lattice of transfer systems on $[n]$ is given by the $(n+1)$-Tamari lattice \cite[Theorem 25]{BBR}. 
\end{example}

\begin{example}
We now know that the transfer systems in Figure \ref{fig:TSN5} form a lattice and not only a poset. 
The transfer systems with a red asterisk are precisely those transfer systems generated by a single edge. These are 

$$\sigma_1 = \includegraphics[scale=0.35]{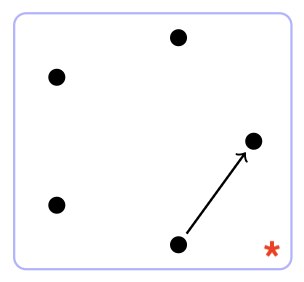}, \quad
\sigma_2 = \includegraphics[scale=0.35]{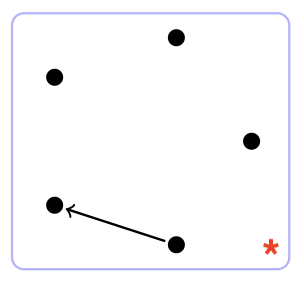} \quad \sigma_3 = \includegraphics[scale=0.35]{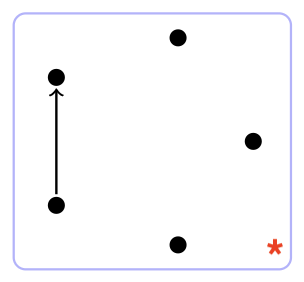} \quad
\sigma_4= \includegraphics[scale=0.35]{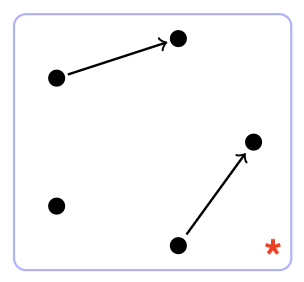},
$$
$$\sigma_5 = \includegraphics[scale=0.35]{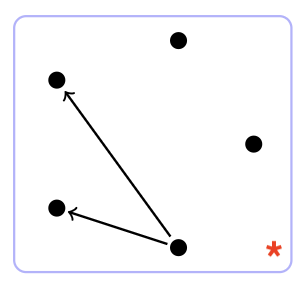}, 
\quad 
\sigma_6 = \includegraphics[scale=0.35]{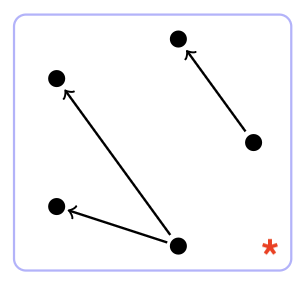},\quad
\sigma_7 = \includegraphics[scale=0.35]{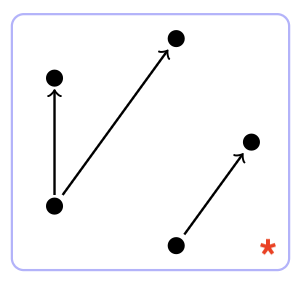},  \quad \sigma_8 = \includegraphics[scale=0.35]{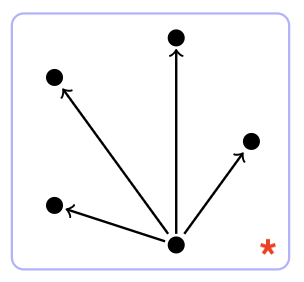}$$

We can then rewrite all the transfer systems on $N_5$ using this set of generators, and we arrive at Figure \ref{fig:gentransfer lattice}.
\begin{figure}[t]
    \centering
    \scalebox{.7}{\input{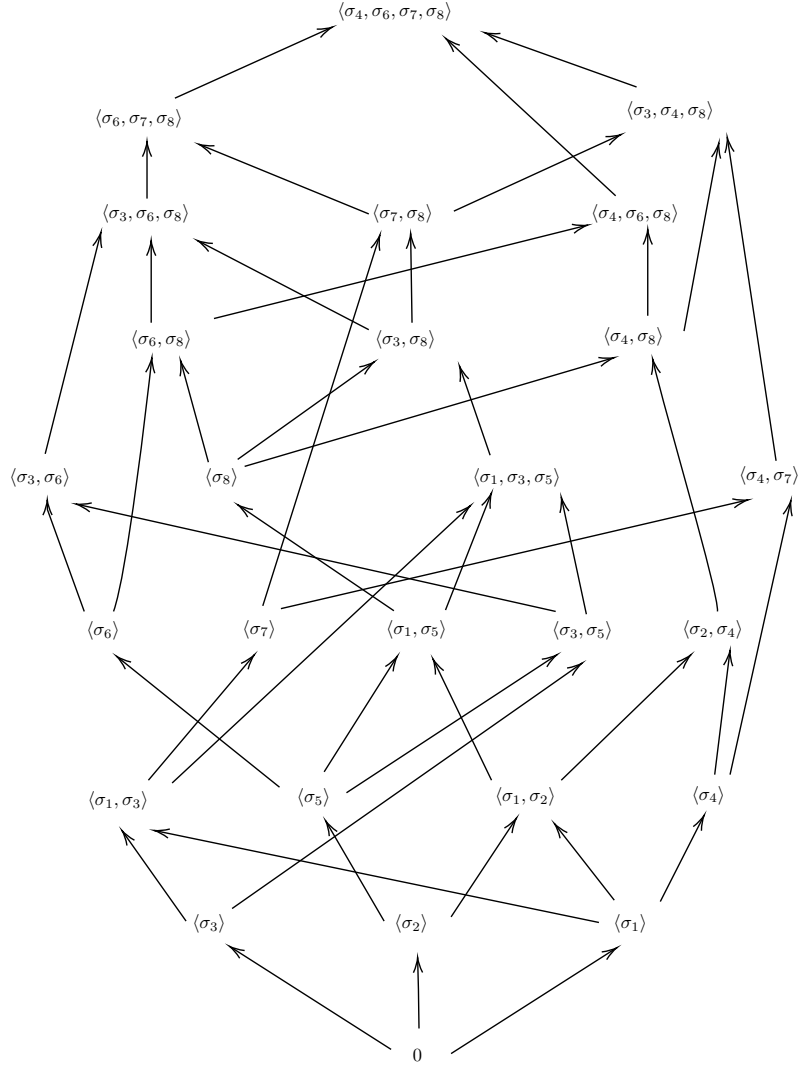}}
    \caption{The lattice $\mathrm{Tr}(N_5)$ using generators.}
    \label{fig:gentransfer lattice}
\end{figure}
\end{example}

\section{Transfer systems and model structures}\label{sec:models}

We assume that the reader is familiar with the basic definitions regarding model categories, but we recall some notions for convenience.

\begin{definition}
A \emph{model structure} on a category $\C$ consists of three subcategories $\W$, $\Cof$ and $\F$ called \emph{weak equivalences}, \emph{cofibrations} and \emph{fibrations}, respectively, satisfying the axioms below. 
The subcategories $\AC = \Cof \cap \W$ and $\AF = \F \cap \W$ are called \emph{acyclic cofbrations} and \emph{acyclic fibrations.}
\begin{enumerate}
\item $\C$ has all finite limit s and colimits.
\item (2-out-of-3) If two out of the three morphisms $f$, $g$ and $g\circ f$ are in $\W$, then so is the third.
\item (Retracts) $\W$, $\Cof$, and $\F$ are closed under retracts. 
\item (Lifts) $\AC={}^\boxslash \F$ and $\Cof = {}^\boxslash \AF$
\item (Factorization) $\W=\AF \circ \Cof = \F \circ \AC$
\end{enumerate}
\end{definition}
We will denote fibrations with a double-tipped arrow.

We make the following observations.
\begin{itemize}
\item In the case of the category $\C$ being a finite lattice, the first two axioms become trivial. 
\item The data a model structure is completely determined by the classes $\W$ and $\AF$ only. The same holds for other choices of pairs of classes, for example $\C$ and $\F$ only, but we choose the presentation of $\W$ and $\AF$ as it is the most suitable for our methods. 
\item The acyclic fibrations $\AF$ are always closed under retracts and contain all identity morphisms \cite[Proposition 3.14]{DS}, which means that they form a transfer system. This is also true for $\F$.

\end{itemize}

Furthermore, when $\C$ is a lattice, the weak equivalences of a model structure satisfy something stronger than the 2-out-of-3 axiom, see \cite[Proposition 1.8]{DZ}.

\begin{lemma}
The weak equivalences of a model structure on a lattice $\C$ are decomposable, i.e. if $g \circ f$ is in $\W$, then so are $g$ and $f$.
\end{lemma}

Therefore, on a finite lattice the data of a model structure gives us a wide decomposable subcategory $\W$ and a transfer system $\T=\AF \subseteq \W$. However, not every wide decomposable subcategory is the weak equivalence set of a model structure. 

\begin{example}
On the lattice $[2]\times[1]$, consider the wide decomposable subcategory with only nontrivial arrow $(1,0) \rightarrow (1,1)$. There is no model structure with this as its weak equivalence set: the factorization axiom tells us that the map $(1,0) \rightarrow (1,1)$ has to be either an acyclic fibration or an acyclic cofibration. Therefore, in particular, either its pullback or its pushout must be a weak equivalence, which is not the case.
\end{example}

In this spirit, we say that a wide decomposable subcategory $\W$ is a \emph{weak equivalence set} if there is a model structure with $\W$ as its weak equivalences.
There is in fact a necessary and sufficient condition for a wide decomposable subcategory to be a weak equivalence set, see \cite[Theorem 5.8]{MORSVZ2}.

\begin{theorem}\label{thm:legalW}
A wide decomposable subcategory $\W$ of a finite lattice $\C$ is a weak equivalence set if and only if following holds:

For every $f \in \W$ there is a factorization into short arrows
$f=\sigma_n \circ \sigma_{n-1} \circ \cdots \circ \sigma_2 \circ \sigma_1$
such that there is a $k$ with
\begin{itemize}
\item for $i \leq k$, all pushouts of $\sigma_i$ are in $\W$,
\item for $i >k$, all pullouts of $\sigma_i$ are in $\W$.
\end{itemize}
\end{theorem}

Let us apply this result to the pentagon $N_5$. 

\begin{proposition}
On the pentagon lattice $N_5$, every wide decomposable subcategory is a weak equivalence set.
\end{proposition}

\begin{proof}
The claim is proven using the condition from Theorem \ref{thm:legalW} by breaking the work up into the possible lengths of a morphism $f \in \W$. (By ``length'' we mean the number of short arrows a morphism can be factored into.) 
Table \ref{tab:pushoutspullbacks} of pushouts and pullbacks in $N_5$ tells us that every morphism in $N_5$ of length 1 has either no nontrivial pushouts or no nontrivial pullbacks, so if $f \in \W$ is of length 1 then the condition of Theorem \ref{thm:legalW} is satisfied automatically. 

There are two morphisms in $N_5$ of length 2, namely $A \rightarrow 1$ and $0 \rightarrow C$. We see that if we factor $A \rightarrow 1$ into short arrows $A \rightarrow C \rightarrow 1$, the two short arrows $A \rightarrow C$ and $C \rightarrow 1$ have no nontrivial pushouts, so the condition of Theorem \ref{thm:legalW}. A similar argument holds for $0 \rightarrow C$. 

The only morphism in $N_5$ of length 3 is $0 \rightarrow 1$. If $0 \rightarrow 1 \in \W$, then by decomposability of $\W$, we automatically have $\W$ is the set of all morphisms in $N_5$ and there is nothing to check. 
\end{proof}

\begin{corollary}
There are 22 weak equivalence sets on $N_5$.
\qed
\end{corollary}

Fixing a weak equivalence set $\W$, not every transfer system $\T \subseteq \W$ arises as the acyclic fibrations of a model category. 

\begin{example}
Consider the lattice $\C=[1]\times [1]$ and the wide decomposable subcategory $\W$ with only nontrivial arrow $(0,0)\rightarrow (0,1)$. Then there is a model structure on $\C$ with $\AF=\W$, but there is no model structure with this $\W$ and the trivial transfer system as its acyclic fibrations. This can be verified directly from the model category axioms, see also \cite[Remark 3.15]{BOOR} and \cite[Example 4.16]{MORSVZ2} for details.
\end{example}

The above example illustrates that the model category axioms always force a certain minimal set of acyclic fibrations to be present. A complete characterization is given by the following result from \cite[Theorem 4.20]{MORSVZ2}.

\begin{theorem}\label{thm:tmin}
For every weak equivalence set $\W$ there is a transfer system $\T_{min}\subset \W$ such that the following are equivalent:
\begin{itemize}
\item There is a model structure with weak equivalences $\W$ and $\AF=\T$.
\item $\T_{min} \subseteq \T \subseteq \W$.
\end{itemize}
\end{theorem}

Fix a weak equivalence set $\W$.
We denote by $\AF(\W) \subseteq \mathrm{Tr}(\mathcal{C})$ the set of all of transfer systems $\T$ form a model structure on $\C$ with weak equivalences $\W$ and acyclic fibrations $\T$.
Theorem 4.20 of  \cite{MORSVZ2} shows that $\AF(\W)$ is a sublattice of $\mathrm{Tr}(\C)$ and moreover is an interval. Recall that an interval in a lattice $\mathcal{P}$ with minimal element $A$ and maximal element $B$ is the set \[
[A,B] = \{ X \in \mathcal{P} \,\,|\,\, A \leq X \leq B \}.
\]
(Despite its name, note that an interval in a lattice is not necessarily isomorphic to $[n]$.)
More precisely, one has
\[
\AF(\W)=[\T_{min},\T_{max}]=\{ \T \in \mathrm{Tr}(\C)\,\,|\,\, \T_{min} \subseteq \T \subseteq \T_{max}\},
\]
where $\T_{max}$ is the (unique) largest transfer system contained in $\W$, and $\T_{min}$ is the transfer system from Theorem \ref{thm:tmin}.

With the previous theorem in mind, for each $\W$ on $N_5$, we can work out what $\T_{min}$ is using the formula
\[
\T_{min}= (\K_{max}^\boxslash) \cap \W,
\]
where $\K_{max}$ is the largest cotransfer system contained in $\W$. This is because for any weak equivalence set $\W$, there is always a largest possible set of corresponding acyclic cofibrations, and this turns out to be $\K_{max}$ \cite[Corollary 4.26]{MORSVZ2}. By duality, $\T_{min}$ is going to be the acyclic fibration set corresponding to $\K_{max}$, which results in the above formula.

We apply this method to the 22 weak equivalence sets of $N_5$
and display the results in the tables Figure \ref{tableAFW3} and \ref{tableAFW} and \ref{tableAFW2}.
The first column shows the weak equivalence sets $\W$ as shaded (red) regions, and the second column is the corresponding lattice $\AF(\W)$.
Enumerating these calculations, we see that there are 70 model structures on $N_5$.

\begin{minipage}[c]{0.45\linewidth}
\begin{table}[H]
    \centering
    \scalebox{.5}{\input{AF_W_table3}}
    \caption{Model Structures on $N_5$, part 1}
    \label{tableAFW3}
\end{table}
\end{minipage}\hfill
\begin{minipage}[c]{0.45\linewidth}
\begin{table}[H]
    \centering
    \scalebox{.5}{\input{AF_W_table}}
    \caption{Model Structures on $N_5$, part 2}
    \label{tableAFW}
\end{table}   
\end{minipage}\hfill
\begin{table}[H]
    \centering
    \scalebox{.5}{\input{AF_W_table2}}
    \caption{Model Structures on $N_5$, part 3}
    \label{tableAFW2}
\end{table}

Looking at the last two weak equivalences sets before $\W=\mbox{ALL}$, we can make the following observation. 

\begin{proposition}\label{prop:afnonmodular} There are nontrivial weak equivalence sets $\W \subseteq N_5$ for which $\AF(\W)$ is nonmodular. Consequently, lattice of transfer systems $\mathrm{Tr}(N_5)$ on $N_5$ is nonmodular.
\end{proposition}

\begin{proof}
For $\W_1= \{A \rightarrow C,\, C \rightarrow 1,\, A \rightarrow 1,\, 0 \rightarrow B \} $ and $\W_2= \{0 \rightarrow A,\, A \rightarrow C,\, 0 \rightarrow C,\, B \rightarrow 1 \}$, we see from Figure \ref{tableAFW2} that the intervals $\AF(\W_1)$ and $\AF(\W_2)$ both contain $N_5$ as a sublattice, namely: 

\begin{figure}[H]\label{tab}
    \centering
    \scalebox{.75}{\input{n5inn5_2}}
    \label{fig:n5n5}
\end{figure}

As $\AF(\W_1)$ and $\AF(\W_2)$ are sublattices of $\mathrm{Tr}(N_5)$, this also implies that $N_5$ is a sublattice of $\mathrm{Tr}(N_5)$.

\end{proof}

\begin{remark}
We could of course directly state that  $\mathrm{Tr}(N_5)$ contains $N_5$ as a sublattice, for example via the examples of $\AF(\W_1)$ or $\AF(\W_2)$. Another instance of $N_5$ as a sublattice of $\mathrm{Tr}(N_5)$ is shown in Figure \ref{fig:penta}.
\begin{figure}[H]
    \centering
    \scalebox{.75}{\input{pentpent}}
    \caption{A copy of $N_5$ inside of $\mathrm{Tr}(N_5)$}
    \label{fig:penta}
\end{figure}
However, as the images of $\mathrm{Tr}(N_5)$ in Figures \ref{fig:TSN5} and \ref{fig:gentransfer lattice} also show, the lattice structure of $\mathrm{Tr}(N_5)$ is complicated, so this is not so easily found from direct inspection.
\end{remark}

The last proposition alludes to a more general statement. If a lattice $\mathcal{P}$ contains $[2]$ as a sublattice, then it is easily verified that the lattice $\mathrm{Tr}(\mathcal{P})$ contains $\mathrm{Tr}([2])\cong N_5$ as a sublattice, see also \cite[Proposition 1.1]{LR24}. 
In particular, if $\mathcal{P}$ is nonmodular, then $\mathrm{Tr}(\mathcal{P})$ is always nonmodular, but as the example $\mathcal{P}=[n]$ shows, the converse does not hold. 
Another notable point for $\mathcal{P}=N_5$ is that when $\AF(\W)$ is nonmodular, then $\W$ is both a transfer and a cotransfer system, but the converse does not hold.

\section{Bousfield localization}\label{sec:bousfield}
Working with Bousfield localization on a finite lattice of course means that we do not need to pay attention to any of the set-theoretic existence questions that usually arise. We refer to \cite[Chapter 7.1]{FOSHT} and \cite[Chapter 3]{Hirschhorn} for the general definitions and only summarize what is required for our purposes.

\emph{Left Bousfield localization $L_S$ at a set $S$} changes an existing model structure $\W, \Cof, \F$ in the following way:
\begin{itemize}
\item $L_S(\W)$ is the smallest weak equivalence set containing $S$, the pushouts of $S$, and $\W$,
\item the cofibrations remain the same, i.e. $L_S(\Cof)=\Cof$.
\end{itemize}
As a consequence, by the lifting axiom of model structures, $L_S(\AF)=\AF$, $L_S(\F) \subseteq F$ and $L_S(\AC) \supseteq \AC$. We will specify what we mean by the ``smallest'' weak equivalence set further along the way.

Dually, \emph{right Bousfield localization $R_S$ at a set $S$} changes an existing model structure $\W, \Cof, \F$ in the following way:
\begin{itemize}
\item $R_S(\W)$ is the smallest weak equivalence set containing $S$, the pullbacks of $S$ and $\W$,
\item the fibrations remain the same, i.e. $R_S(\F)=\F$.
\end{itemize}
Consequently, $\AC = R_S(\AC)$, $\AF \subseteq R_S(\AF)$ and $\Cof \supseteq R_S(\Cof)$. 

Let us specify how the weak equivalences on an underlying lattice change under left and right localization. We concentrate on the case where the set $S$ consists of a single morphism $f$. For both left (and right) localization, we start by adding $f$ to $\W$ and then keep adding the necessary morphisms until we reach a set containing $\W$, $f$, and the pushouts (resp. pullbacks of $f$), and which satisfies the conditions of Theorem \ref{thm:legalW}. We only describe $R_f(\W)$; $L_f(\W)$ is constructed dually using $\AC$ instead of $\AF$ and pushouts instead of pullbacks. More details for both constructions are given in \cite[Section 2]{CGMNRT}.

\pagebreak

\begin{construction}\label{con:wchanges}
{From $\W$ to $R_f(\W)$}:

\begin{enumerate}
\item Establish $\AF_{(0)} := \AF, \,\,\, S_{(0)}:=\{f\}$.
\item Define $\AF_{(n)} := \langle \AF_{(n-1)} \cup S_{(n-1)} \rangle_{TS}$.
\item Take $\W_{(n)}$ to be the closure of $\AF_{(n)} \circ \AC$ under the two-out-of-three property.
\item Set $S_{(n)}$ to be the set of arrows in $\W_{(n)} \backslash \W_{(n-1)}$ and continue with Step (2).
\end{enumerate}
This process either terminates or stabilizes as we are in a finite lattice, and the resulting set is the weak equivalence set $\W_N=R_f(\W)$ for some index $N$.
\end{construction}

Since we present the data of a model category in terms of its weak equivalences and acyclic fibrations, with this construction, we know how the weak equivalences change for both left and right localization. For left localization, we also know that the acyclic fibrations remain unchanged. 
For right localization, the acyclic fibrations are changed by the addition of the following set, see \cite[Theorem 2.16]{CGMNRT}

\begin{theorem}\label{thm:goldenarrows}
For right Bousfield localization on a lattice at a short arrow $f$, the acyclic fibrations become $R_f(\AF)=\left<\AF \cup \Gamma_f\right>_{TS}$
where 
    \[
\Gamma_f = \{    s \rightarrow t \, \vert \, s \in S(\sigma) , t \in T(\sigma) \,\,\mbox{for some short arrow $\sigma \in R_f(\W)\setminus \W$}  \}
\]
with 
\begin{eqnarray}
T(\sigma) = \{ y \,\,|\,\, \mbox{$y$ is weakly equivalent to the target of $\sigma$ before localization}, \nonumber\\ \mbox{and $y$ is maximal with this property}\} \nonumber
\end{eqnarray}
and 
\begin{eqnarray}
S(\sigma)=\{y \,\,|\,\, \mbox{$y$ is weakly equivalent to the source of $\sigma$ before localization}, \,\, y \leq t\,\, \nonumber\\ \mbox{for some $t \in T(\sigma)$, and $y$ is maximal with these properties}\}. \nonumber
\end{eqnarray} 
The set $\Gamma_f$ is called the \emph{Golden Arrows}.
\end{theorem}

In other words, we look at the weak equivalence components before localization, observe which of those are joined after localization and add arrows from the maximal elements of the source components to the maximal elements of the target components with respect to the old weak equivalences.

\begin{example} To visualize this result, we look at Example 2.13 from \cite{CGMNRT}.
Consider the following model structure on the lattice $[1] \times [1]$, where the weak equivalences are given by the shaded area and every morphism is a fibration, and thus $\W=\AF$.
This is the only model structure on $[1] \times [1]$ with this weak equivalence set. 
\begin{figure}[H]

\begin{tikzpicture}[scale=1.1]
\draw[red!30, fill, rounded corners] (0.6, 0.6) rectangle (2.4, 1.4); 
\draw[red!30, fill, rounded corners] (1.6, 1.6) rectangle (2.4, 2.4); 
\draw[red!30, fill, rounded corners]  (0.6, 0.6)rectangle (1.4,2.4); 
 \foreach \x in {1,2}
 \foreach \y in {1,2}
    \fill (\x,\y) circle (1mm);
    \node at (1,2) (Sf) {}; 
    \node at (2,2) (Tf) {}; 
    \node at (2,1) (B) {};
    \node at (1,1) (C) {};
  \draw[line width=.5mm, black, -{stealth}] (Sf) edge node[above]{$f$} (Tf);
   \draw[thick, ->>, black] (C) edge (Sf);
    \draw[thick, ->>, black] (C) edge (B);
  
\end{tikzpicture}
\end{figure}
We would like to right localize at $f: (1,0) \rightarrow (1,1)$. If $f$ becomes a weak equivalence, then the 2-out-of-3 axiom implies that $R_f\W=\mbox{ALL}.$

The two shaded weak equivalence components get joined in localization, so the Golden Arrows go from the maximal vertices of the first component to the maximal (and only) element of the second component. 
The Golden Arrows are denoted by $g_1$ and $g_2$ in the picture below.
\begin{figure}[H]
\hspace{0.25cm}
\begin{tikzpicture}[scale=1.1]
\draw[red!30, fill, rounded corners] (0.6, 0.6) rectangle (2.4, 1.4); 
\draw[red!30, fill, rounded corners] (1.6, 1.6) rectangle (2.4, 2.4); 
\draw[red!30, fill, rounded corners]  (0.6, 0.6)rectangle (1.4,2.4); 
 \foreach \x in {1,2}
 \foreach \y in {1,2}
    \fill (\x,\y) circle (1mm);
    \node at (1,2) (Sf) {}; 
    \node at (2,2) (Tf) {}; 
    \node at (2,1) (Tp) {};
      \node at (2,1) (B) {};
    \node at (1,1) (C) {};
   \draw[thick, ->>, black] (C) edge (Sf);
    \draw[thick, ->>, black] (C) edge (B);
    \draw[line width=.5mm, black, -{stealth}] (Sf) edge node[below]{$f$} (Tf);
    \draw[line width=2.5pt, gold, -{stealth}, dashed] (Sf) edge[bend left] node[above]{\color{black}$g_1$} (Tf);
    \draw[line width=2.5pt, gold, -{stealth},dashed] (Tp) edge[bend right] node[right]{\color{black}$g_2$} (Tf);
\end{tikzpicture}
\end{figure}

Therefore, the new acyclic fibrations are obtained by adding the Golden Arrows to the old acyclic fibrations and closing the result under transfer system operations, giving
\[
R_f(\AF)=\left<\AF \cup \Gamma_f\right>_{TS}=\mbox{ALL}.
\]
\end{example}

We will now study how the model structures on $N_5$ given in Tables \ref{tableAFW3} and \ref{tableAFW} are interlinked by Bousfield localization. 
One key question about all Bousfield localizations on a fixed underlying lattice is if the poset of all model structures together with the left and right localizations is connected. In other words, can any model structure be obtained from the trivial model structure by a sequence of left and right localizations?

For the 23 model structures on the square $[1] \times [1]$, one can observe that the answer is positive by direct calculation \cite[Figure 5]{CGMNRT}. By \cite[Theorem 5.12]{BOOR}, this is also true for any total order $[n]$. 
For an arbitrary grid $[n] \times [m]$, the answer is already more complex. It was shown in \cite[Section 4]{CGMNRT} that the claim is false in general, but that it still holds when restricted to \emph{saturated} transfer systems (i.e. those transfer systems satisfying the 2-out-of-3 condition). For our nonmodular lattice $N_5$ we reach the following conclusion.

\begin{theorem}\label{thm:allthelocals}
Every model category structure on $N_5$ is a sequence of left and right Bousfield localizations of the trivial model structure.
\end{theorem}

\begin{proof}
We verify this by direct calculation, using Construction \ref{con:wchanges} to see the changes in weak equivalences, and the Golden Arrows from Theorem \ref{thm:goldenarrows} to see the changes in the transfer system of acyclic fibrations. We note again that left localization keeps the acyclic fibrations as they are. We also use Tables \ref{tableAFW3} and \ref{tableAFW} to keep track of our workings. We will arrive at the diagram Figure \ref{fig:tree}
at the end of this proof. 

Again, we depict a model structure as a pair $\W \supset \AF$. We display the weak equivalences $\W$ as the shaded areas, and the transfer system $\AF \subset \W$ as the double headed arrows. We denote left localizations by wavy arrows and right localizations by solid arrows. 

Rather than only presenting the final image as the proof, we will talk the reader through its inception and verification.
In practice, each arrow from a model structure $\W\supset\AF$ to a $L_f(\W)\supset L_f(\AF)$ or a $R_f(\W)\supset R_f(\AF)$ is not too difficult to calculate, especially as Construction \ref{con:wchanges} does not have to be iterated in our case- the main challenge is indeed the bookkeeping. 

First, we notice that it is only necessary to include localizations at short arrows in the diagram, as localization at decomposable arrows can always be written as composition of localization at short arrows. 
We also do not draw all possible localization arrows into the diagram to keep it as visually simple as possible for our purposes.
Next, we can simplify our work by proceeding with localizing only at those arrows that do not have any nontrivial pushouts (for left localization) or pullbacks (for right localization) that are not already in $\W$. 

For left localizations, we add a short arrow and its pushouts to $\W$ and close under the 2-out-of-3 property. The transfer systems of acyclic fibrations does not change. For right localization, we add a short arrow and its pullbacks to $\W$. The transfer system of acyclic fibrations changes via the addition of the Golden Arrows.
To illustrate how this works, we present a couple of sample calculations.

Consider the model structure in the first image, with the weak equivalences depicted as the shaded regions and the acyclic fibrations drawn as the double-tipped arrows. Left or right localizing at $f: C \rightarrow 1$ adds $f$ to the weak equivalences, so closing under the 2-out-of-3 property we end up with $\W=\mbox{ALL}$ after localizing. The transfer system of acyclic fibrations remains unchanged under left localization.

 \begin{figure}[h]
 
\begin{tikzpicture} 
\draw[red!30, fill, rounded corners] (1.126, 1.663) rectangle (1.726, 2.263);
  \draw[red!30, fill, rounded corners]  (-0.3,-0.15) -- (-0.3,1.25+0.85) -- (1.4265848+0.3,-0.4635254-0.3)--cycle;
    \draw[red!30, fill, rounded corners] (2.3082627-0.1, 0.75+0.3)-- (2.3082627+0.5, 0.75+0.2)-- (1.4265848+0.2, -0.4635254-0.3) -- (1.4265848-0.4, -0.4635254-0)-- cycle;
    \fill (0,0) circle (1mm);
    \fill (0,1.5) circle (1mm);
    \fill (1.4265848,1.9635254) circle (1mm);
    \fill (2.3082627, 0.75) circle (1mm);
    \fill (1.4265848,-0.4635254) circle (1mm);  
\node (X) at (1.4265848,-0.4635254) {};
\node (B) at (2.3082627, 0.75) {};
\node (A) at (0,0) {};
\node (C) at (0,1.5) {};
\node (Y) at (1.4265848,1.9635254) {};
\draw[thick, ->>](X) edge (A);
\draw[thick, ->>](X) edge (B);
\draw[thick, ->>](X) edge (C);
\draw[dashed, thick, ->](C) edge node[above]{$f$} (Y);
\node (S) at (2.3082627+1, 0.75) {};
\node (J) at (2.3082627+3, 0.75) {};
\draw[thick, ->](S) edge node[above]{$L_f$} (J);
\end{tikzpicture}
\begin{tikzpicture} 
 \draw[red!30, line width=2pt, rounded corners]  (-0.3,-0.15) -- (-.3,1.5+0.2) -- (1.4265848+0.1,1.9635254+ 0.3)-- (2.3082627+0.3, 0.75)-- (1.4265848+0.15,-0.4635254-0.3)--cycle;
    \fill (0,0) circle (1mm);
    \fill (0,1.5) circle (1mm);
    \fill (1.4265848,1.9635254) circle (1mm);
    \fill (2.3082627, 0.75) circle (1mm);
    \fill (1.4265848,-0.4635254) circle (1mm);  
\node (X) at (1.4265848,-0.4635254) {};
\node (B) at (2.3082627, 0.75) {};
\node (A) at (0,0) {};
\node (C) at (0,1.5) {};
\node (Y) at (1.4265848,1.9635254) {};
\draw[thick, ->>](X) edge (A);
\draw[thick, ->>](X) edge (B);
\draw[thick, ->>](X) edge (C);
\end{tikzpicture}
  \end{figure}
     
If we right localize instead of left localize the first model structure at $f$, we again obtain $\W=\mbox{ALL}$, but the transfer system $\AF$ changes by the addition of the Golden Arrows indicated in the second image below. Closing under pullbacks and composition gives us the transfer system $R_f(\AF)$ drawn in the third image.
 
\begin{figure}[h]
\begin{tikzpicture} 
\draw[red!30, fill, rounded corners] (1.126, 1.663) rectangle (1.726, 2.263);
  \draw[red!30, fill, rounded corners]  (-0.3,-0.15) -- (-0.3,1.25+0.85) -- (1.4265848+0.3,-0.4635254-0.3)--cycle;
    \draw[red!30, fill, rounded corners] (2.3082627-0.1, 0.75+0.3)-- (2.3082627+0.5, 0.75+0.2)-- (1.4265848+0.2, -0.4635254-0.3) -- (1.4265848-0.4, -0.4635254-0)-- cycle;
    \fill (0,0) circle (1mm);
    \fill (0,1.5) circle (1mm);
    \fill (1.4265848,1.9635254) circle (1mm);
    \fill (2.3082627, 0.75) circle (1mm);
    \fill (1.4265848,-0.4635254) circle (1mm);  
\node (X) at (1.4265848,-0.4635254) {};
\node (B) at (2.3082627, 0.75) {};
\node (A) at (0,0) {};
\node (C) at (0,1.5) {};
\node (Y) at (1.4265848,1.9635254) {};
\draw[thick, ->>](X) edge (A);
\draw[thick, ->>](X) edge (B);
\draw[thick, ->>](X) edge (C);
\draw[dashed, thick, ->](C) edge node[above]{$f$} (Y);
\node (S) at (2.3082627+1, 0.75) {};
\node (J) at (2.3082627+2.5, 0.75) {};
\draw[thick, ->](S) edge node[above]{$\Gamma_f$} (J);
\end{tikzpicture}
\begin{tikzpicture} 
 \draw[red!30, line width=2pt, rounded corners]  (-0.3,-0.15) -- (-.3,1.5+0.2) -- (1.4265848+0.1,1.9635254+ 0.3)-- (2.3082627+0.3, 0.75)-- (1.4265848+0.15,-0.4635254-0.3)--cycle;
    \fill (0,0) circle (1mm);
    \fill (0,1.5) circle (1mm);
    \fill (1.4265848,1.9635254) circle (1mm);
    \fill (2.3082627, 0.75) circle (1mm);
    \fill (1.4265848,-0.4635254) circle (1mm);  
\node (X) at (1.4265848,-0.4635254) {};
\node (B) at (2.3082627, 0.75) {};
\node (A) at (0,0) {};
\node (C) at (0,1.5) {};
\node (Y) at (1.4265848,1.9635254) {};
\draw[thick, ->>](X) edge (A);
\draw[thick, ->>](X) edge (B);
\draw[thick, ->>](X) edge (C);
\draw[line width=2.5pt, gold, -{stealth}, dashed](B) edge (Y);
\draw[line width=2.5pt, gold, -{stealth}, dashed](C) edge (Y);
\node (S) at (2.3082627+1, 0.75) {};
\node (J) at (2.3082627+2.5, 0.75) {};
\draw[thick, ->](S) edge node[above]{$\langle \AF \cup \Gamma_f \rangle_{TS}$} (J);
\end{tikzpicture}
\begin{tikzpicture} 
 \draw[red!30, line width=2pt, rounded corners]  (-0.3,-0.15) -- (-.3,1.5+0.2) -- (1.4265848+0.1,1.9635254+ 0.3)-- (2.3082627+0.3, 0.75)-- (1.4265848+0.15,-0.4635254-0.3)--cycle;
    \fill (0,0) circle (1mm);
    \fill (0,1.5) circle (1mm);
    \fill (1.4265848,1.9635254) circle (1mm);
    \fill (2.3082627, 0.75) circle (1mm);
    \fill (1.4265848,-0.4635254) circle (1mm);  
\node (X) at (1.4265848,-0.4635254) {};
\node (B) at (2.3082627, 0.75) {};
\node (A) at (0,0) {};
\node (C) at (0,1.5) {};
\node (Y) at (1.4265848,1.9635254) {};
\draw[thick, ->>](X) edge (A);
\draw[thick, ->>](X) edge (B);
\draw[thick, ->>](X) edge (C);
\draw[thick, ->>](X) edge (Y);
\draw[thick, ->>](B) edge (Y);
\draw[thick, ->>](C) edge(Y);
\end{tikzpicture}

    \label{fig:placeholder}
\end{figure}

Putting all those calculations together, we arrive at the following.

\begin{figure}[H]\centering \scalebox{.55}{\input{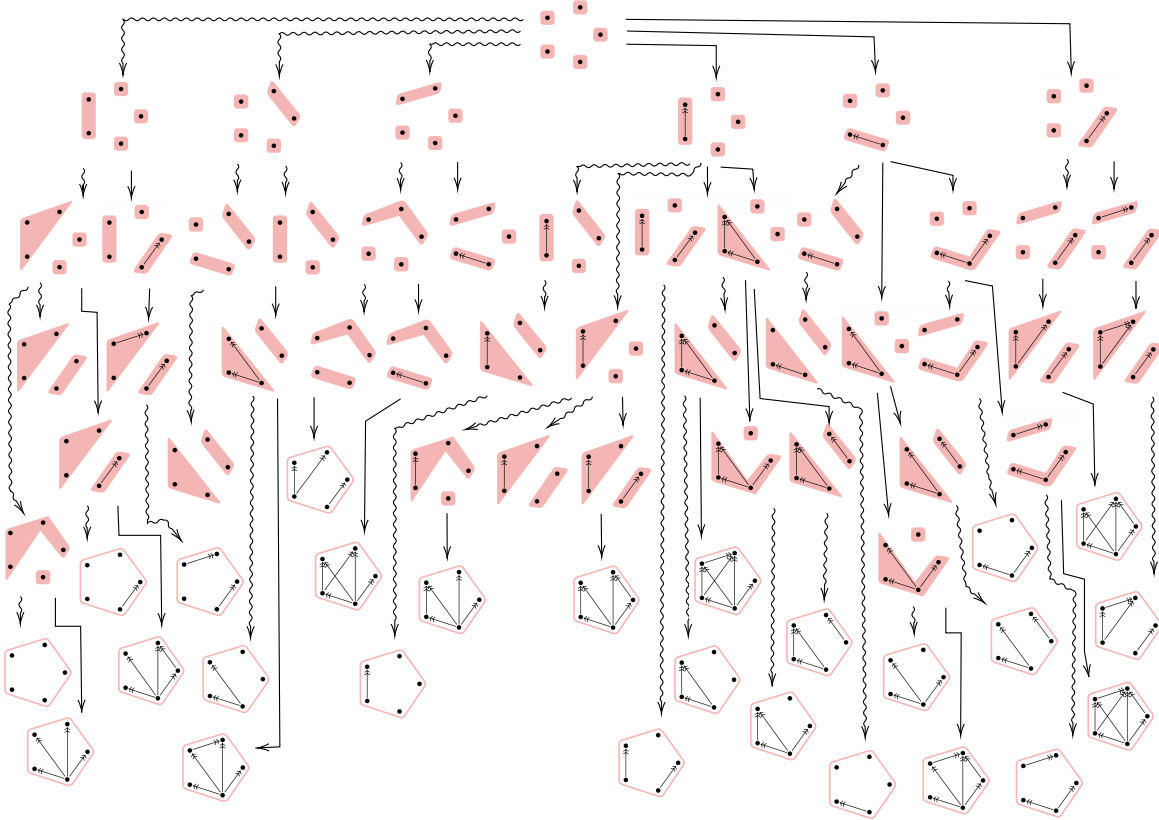}} \caption{All 70 model structures of $N_5$ can be obtained from the trivial model structure via a sequence of left and right Bousfield localization.}\label{fig:tree} \end{figure}

\end{proof}

We can therefore conclude that the nonmodular lattice $N_5$ is well-behaved in its homotopy theory, which initially was a surprise as nonmodularity often presents itself as a practical challenge. In particular, we knew that the property proved in Theorem \ref{thm:allthelocals} is true for the square $[1] \times [1]$ but not the next larger grid $[2] \times [1]$, and it will be interesting to consider how $N_5$ fits into this context. Further, we do not have any indication at this point whether the theorem will hold in full for other nonmodular lattices, or hold with additional restrictions, such as only considering saturated transfer systems as $\AF$.


\newcommand{\etalchar}[1]{$^{#1}$}

\end{document}